\documentclass[11pt]{article}
\usepackage{amsfonts}
\usepackage{latexsym,amssymb}
\usepackage{amsmath, amsbsy}
\usepackage{amsopn, amstext}
\usepackage{graphicx, color, epstopdf}
\usepackage{threeparttable}

\newtheorem{lemma}{Lemma}
\newtheorem{theorem}{Theorem}
\newtheorem{remark}{Remark}

\newtheorem{proposition}{Proposition}

\numberwithin{equation}{section}
 \numberwithin{Lem}{section}
 \numberwithin{Defi}{section}
 \numberwithin{Theo}{section}
 \numberwithin{Rem}{section}
  \numberwithin{Coro}{section}
  \numberwithin{Fig}{section}

\def\NN{\hbox{\rlap{I}\kern.16em N}}
\def\NC{\hbox{\rlap{\kern.24em\raise.1ex\hbox
                  {\vrule height1.3ex width.9pt}}C}}

\title{ Unconditionally optimal convergence of an energy-conserving and linearly implicit scheme for nonlinear wave equations\thanks{This work is supported in part by the NSFC
(Grant Nos. 11771128, 11871106, 11871092, 11926356) and NSAF (Grant No. U1930402).}}
\author{ Waixiang Cao\thanks{ School of  Mathematical Sciences, Beijing  Normal University, Beijing 100875, China ({\tt caowx@bnu.edu.cn})}
\and Dongfang Li\thanks{School of Mathematics and Statistics,
Huazhong University of Science and Technology, Wuhan 430074, China
({\tt dfli@hust.edu.cn}); and Hubei Key Laboratory of Engineering Modeling and Scientific Computing, Huazhong University of Science and Technology, Wuhan 430074, China}
\and Zhimin Zhang\thanks{Beijing Computational Science Research Center, Beijing 100193, China ({\tt zmzhang@csrc.ac.cn}); and
Department of Mathematics, Wayne State University, Detroit, MI 48202, USA (\tt zzhang@math.wayne.edu)}}
\date{}

\begin{document}

\maketitle

\begin{abstract}
{In this paper, we present and analyze an energy-conserving and linearly implicit scheme for solving the nonlinear wave equations. Optimal error estimates
in time and superconvergent error estimates in space are  established without time-step dependent on the spatial mesh size.
The key is to estimate directly the solution bounds in the $H^2$-norm for
both the nonlinear wave equation and the corresponding fully discrete scheme, while the previous investigations rely on the temporal-spatial error splitting approach. Numerical examples are presented to confirm energy-conserving properties, unconditional convergence, and optimal error estimates, respectively, of the proposed fully discrete schemes.}


\end{abstract}

\section{Introduction}
   We present an energy-conserving and linearly implicit scheme as well as the unconditionally optimal error estimates for solving  the following  wave equation
  \begin{eqnarray}\label{con_laws}
\begin{aligned}
   &u_{tt}=\triangle u-\lambda u-F'(u),\ \ &&({\bf x},t)\in \Omega \times(0,T], \\
   &u({\bf x},0)=u_0({\bf x}),\ \ u_t({\bf x},0)=u_1({\bf x}),\ \  &&{\bf x}\in \Omega
\end{aligned}
\end{eqnarray}
with the periodic boundary condition, where  $\lambda\ge 0$ is a constant, $\Omega$ is a polygonal or polyhedral domain in ${\mathbb  R}^d (d=2,3)$, $u_0$ and $u_1$ are sufficiently smooth,   and $F\in C^2(\mathbb R)$ is the nonlinear potential. For simplicity, we assume that
  $\Omega$ is a rectangular or cubic domain.
  Nonlinear wave equations are widely used to model plenty of
 complicated natural phenomena
in variety of scientific fields \cite{bis09,dod82,dra89,waz08}. In the past several decades, it has been one of the hot spots in the numerical analysis of different schemes for the equations \cite{bru15,cao17,gau15,wan181,wu16}.

   There are many papers that consider error analysis of the fully discrete schemes for the nonlinear problems under the following assumption (e.g., \cite{can10,car04,den77,gri05}),
  \begin{equation}\label{assm1}
    | F'(u^n)-F'(u_h^n)|\leq L|u^n-u_h^n|,
  \end{equation}
where  $u^n$ and $u_h^n$ are respectively the theoretical and numerical solutions, and $L>0$ is the Lipschitz coefficient.
A classical model satisfying \eqref{assm1} is the Sine-Gorden equation, whose nonlinear term is $\sin(u)$. However, as pointed in \cite{gau15}, the assumption \eqref{assm1} is not the typical behavior of
the general nonlinear wave equations and thus  its range of the actual applicability is limited.

In order to deal with the non-Lipschitz nonlinearity, one common way is to
 impose a priori boundedness of the numerical approximations $u_h^n$. In classical finite element analysis, the numerical solutions in the maximum norm are usually estimated  by
\begin{eqnarray} \label{tsr}
\| u^n_h\|_{L^\infty}&\leq &\| R_hu^n\|_{L^\infty}+\| R_hu^n-u^n_h\|_{L^\infty}\nonumber\\
 &\leq &\| R_hu^n\|_{L^\infty}+C h^{-d/2} \| R_h u^n - u^n_h\|_{L^2} \nonumber\\
 &\leq & \| R_hu^n\|_{L^\infty}+ Ch^{-d/2}(\tau^p+h^{r+1}),
\end{eqnarray}
where $R_h$ is the projection operator,  $r+1$ and $p$  are convergence orders in spatial and temporal directions, respectively. Consequently, a time-step restriction $\tau = O(h^{\frac{d}{2p}})$ is needed in \eqref{tsr}(e.g., \cite{can88,Farago,Garcia,Luskin,Rachford}). Such time-step restriction appears widely in the numerical analysis but is not always necessary in actual applications.

Unconditional convergence means that the established error bound is valid without the above mentioned time-step restriction. To achieve the unconditional convergence,  a temporal-spatial error splitting approach is presented recently \cite{lib14,lib13,lib131,li183,li19}. The key idea of the approach
is to introduce a time discrete system, whose solution is denoted by $U^n$. Then, one can get the following error estimates
\[ \|U^n\|_{H^2}\leq C~~~\textrm{and}~~~\|R_hU^n-u_h^n\|_{L^2}\leq Ch^2. \]
The boundedness of the numerical solutions is obtained by {
\begin{eqnarray} \label{tsr1}
\| u^n_h\|_{L^\infty}&\leq &\| R_hU^n-u^n_h\|_{L^\infty}+\| R_hU^n-{\cal I}_hU^n\|_{L^\infty}+\| {\cal I}_hU^n\|_{L^\infty}\nonumber\\
 &\leq &C h^{-d/2}( \| R_h U^n - u^n_h\|_{L^2}+\| R_h U^n - {\cal I}_h U^n\|_{L^2} )+ \| U^n\|_{L^\infty}\nonumber\\
 &\leq &Ch^{2-d/2}+\|U^n\|_{H^2}.
\end{eqnarray}
   Here ${\cal I}_hU^n$ denotes the interponation function of $U^n$.}
It implies the numerical solutions is bounded if the temporal and spatial step-sizes are sufficiently small, respectively. Then,
the error estimates can be proved following the usual way. In spite of the interesting and instructive work, additional error estimates in different norms are required in the proof, and so far, most unconditional convergence results are focused on nonlinear parabolic problems.

Nonlinear wave equations \eqref{con_laws} have several remarkable features. First,
the models \eqref{con_laws} are energy-conserving, i.e.,
\[
E(t)= \int_\Omega u_t^2+ |\nabla u|^2 +\lambda |u|^2 + 2F(u) d x =E(0).
\]
Second, the typical nonlinear terms are non-Lipschitz continuous. Third, the solutions have different regularities.
A natural question is whether we can develop some effective unconditional convergence numerical schemes for nonlinear wave equations, taking all the remarkable features into account.

In the present paper, we present an energy-conserving and linearly implicit scheme for solving the nonlinear wave equations \eqref{con_laws}. The scheme
is of order $2$ in the temporal directions and no additional initial iterations are required.
 The construction of the scheme is based on
recently-developed  scalar auxiliary variable (SAV) approach combined with  finite element methods,
classical Crank-Nicolson methods,  and extrapolation approximation. We show
that our fully-discrete schemes conserve the energy and converge without any temporal step-size restrictions dependent on the spatial step-size.
Unlike the previous temporal-spatial error splitting approach, we estimate the solution directly in the following procedure: (1) obtain the bounds in the $H^2$-norm of the solutions for
both the nonlinear wave equations and the corresponding fully discrete schemes; (2) establish the bound for numerical solutions by applying the embedding inequality; (3) obtain the unconditionally optimal error estimates
in time and superconvergent error estimates in space.

We remark that the key to construct the energy-conserving schemes is the SAV idea, which has been applied successfully  to the gradient flows \cite{akr19,shen-xu,she18,she19}.
Very recently, the idea was introduced to develop energy-conserving schemesfor the conservative laws \cite{cai19,cai-shen,lisun20,li21}. However, much attention has been paid to the stability
and energy-conserving properties, and no unconditional convergence results  of fully discrete SAV schemes for nonlinear wave equations are found in the literature. This is the main motivation and
contribution of the present study.

The rest of paper is organized as follows. In Section \ref{method1}, we propose a fully discrete
scheme for the nonlinear wave equation \eqref{con_laws}. In Section \ref{anal-1}, we present a detailed proof to show the energy-conserving properties and unconditional convergence for the temporal discretization.  Error estimates for the fully discrete solution
is established in Section \ref{error-es},  where we prove that the approximation error
 is unconditionally optimal  in time and superconvergent in space (under the $H^1$ norm).
In Section \ref{num}, we present several numerical examples to confirm the theoretical results. Finally, conclusions are presented in Section \ref{conc}.

   \section{The linearly implicit method }\label{method1}
 In this section, we present a fully discrete numerical scheme, which preserves
the discrete energy.

Suppose $E_1(u)=\int_{\Omega} F(u) d{\bf x}\ge -c_0$ for some $c_0>0$, i.e., it is bounded from below, and let
   $C_0>c_0$ so that $E_1(u)+C_0>0$.
 We introduce the following scalar  auxiliary variable (SAV)
 \[
    r(t)=\sqrt {E(u)},\ \ E(u)=\int_{\Omega} F(u) d{\bf x}+ C_0,
 \]
   and rewrite \eqref{con_laws} as
\begin{eqnarray}\label{con_laws1}
\begin{aligned}
    & u_t=v, &\\
    &v_t=\triangle u-\lambda u -\frac{r}{ \sqrt{ E(u)}}f(u), &\\
    & r_t=\frac{1}{2\sqrt{ E(u)}}\int_{\Omega} f(u)u_t d{\bf x}, &
\end{aligned}
\end{eqnarray}
  where $f(u)=F'(u)$.

  Let ${\mathcal T_h}$ be the usual regular triangulation of the polygonal domain $\Omega$. Denote by $h_{T}$ the meshsize of $\mathcal T_h$, where $h_{T}$ is the diameter of the element ${T} \in \mathcal T_h$, and $h= \max_{  T \in \mathcal T_h}$.
   Let $V_h$ be the classical finite-dimensional subspace of $H^1(\Omega)$, which consists of the usual continuous piecewise polynomials of degree $k$ ($k\geq 1$) on $\mathcal{T}_h$.  That is,
\[
   V_h=\{v\in C^0(\Omega):  v|_{T}\in {\mathbb P}_{k},\ \forall T\in {\cal T}_h\},
\]
  where  ${\mathbb P}_{k}$ denotes  the space of polynomials of degree  no more than $k$.

Let $\tau=\frac{T}{N}$ with $N$ a given integer and $t_n=n\tau$, $n=0,1,\cdots,N$. Denote
\[
    u^n=u({{\bf x}},t_n),\ \ v^n=v({{\bf x}},t_n), \ \ r^n=r(t_n).
\]
  For any sequence of the functions $\{f^n\}_{n=0}^N$, we define for all $n, n=0,\ldots N-1$
 \[
    D_{\tau}f^{n+1}:=\frac{Df^{n+1}}{\tau}=\frac {f^{n+1}- f^{n}}{\tau},\ \  \tilde f^{n+\frac 12}:=\frac 12 (3f^{n}-f^{n-1}), \ \  \hat f^{n+\frac 12}:=\frac{f^{n+1}+f^{n}}{2}.
 \]
    Note that for $n=0$,  we  denote by  $ \tilde f^{\frac 12}=f^0$.

 { To design an energy-conserving and linearly implicit numerical  scheme,  which is easy to implement and efficient,   we consider the following} fully discrete Crank-Nicolson Galerkin SAV method:  Find $u^{n+1}_h\in V_h, v^{n+1}_h\in V_h, r_h^{n+1}\in  \mathbb R$ for $n=0,\ldots N-1$ such that for all $(w_h,\zeta_h)\in V_h\times V_h$
 \begin{eqnarray}\label{sav}
 \begin{split}
    & \big(D_{\tau}u_h^{n+1}, w_h\big)=\big(\hat v_h^{n+\frac 12}, w_h\big), &&\\
    &\big(D_{\tau}v_h^{n+1}, \zeta_h\big)=-\big( \nabla \hat u_h^{n+\frac 12}, \nabla \zeta_h\big)-\big(\lambda \hat u_h^{n+\frac 12}, \zeta_h\big)
    -\big( \frac{\hat r_h^{n+\frac 12}}{\sqrt {E(\tilde u_h^{n+\frac 12})}}f(\tilde u_h^{n+\frac 12}), \zeta_h \big), &&\\
    &  {r_h^{n+1}- r_h^{n}} =\frac{1}{2\sqrt {E(\tilde u_h^{n+\frac 12})}}\int_{\Omega} f(\tilde u_h^{n+\frac 12}) ({u_h^{n+1}-u_h^n})  d{\bf x}, &&
 \end{split}
\end{eqnarray}
  where $(u,v)=\int_{\Omega} u ({\bf x}) v({\bf x}) d{\bf x}$, $f(\tilde u_h^{\frac 12})=f(u_h^0)$, and initial values are chosen as $${ (u_h^0,v_h^0,r_h^0)=(R_h u_0,R_h u_1,\sqrt{E(u_0)})}.$$
  Here $R_hu_0$ is the Ritz projection of $u_0$, which will be defined later.

      Equivalently,  we rewrite the above scheme \eqref{sav}   into the following linear form
\begin{equation}\label{eq:11}
    ((4I-\tau^2 {\triangle_h}{ +}\tau^2\lambda I)u_h^{n+1}, w_h)+\frac{\tau^2}{2} (u_h^{n+1},b_1)\big(b_1,  w_h \big)=(g, w_h)+\frac{\tau^2}{2} (u_h^{n},b_1)\big(b_1,  w_h \big)
\end{equation}
  for  all $ w_h\in V_h$,  where  {  $( \triangle_h u_h,v_h):=-(\nabla u_h,\nabla v_h)$, and }
\[
   b_1= \frac{f(\tilde u_h^{n+\frac 12})}{\sqrt {E(\tilde u_h^{n+\frac 12})}},\ \  g=(4I{ +}\tau^2{ \triangle_h}-\tau^2\lambda I)u_h^{n}+4\tau v_h^n-2\tau^2 r^n_hb_1.
\]
  Choosing $w_h$ in \eqref{eq:11} to be the basis function of $V_h$  leads to a linear equation of the form
 \[
    A{\bf u}^{n+1}+ ({\bf u}^{n+1}, {\bf b}_1){\bf b}_2={\bf g}
 \]
   for some matrix $A$ and vectors ${\bf b}_1,{\bf b}_2, {\bf g}$. By taking the inner product with ${\bf b}_1$ in the above equation,   we obtain
$({\bf u}^{n+1}, {\bf b}_1) $ and then derive  ${\bf u}^{n+1}$. Hence the scheme is  easy to implement and very efficient. We also
refer to  \cite{she18,she19} for more detailed information.



\section{Unconditionally energy preservation and convergence for the temporal discretization}\label{anal-1}

   In this section, we shall prove that the Galerkin SAV method \eqref{sav} preserves the energy
   unconditionally. Moreover, we establish the convergence analysis of the SAV approach with minimum assumptions.

  We begin with the energy preservation property of the Galerkin SAV approach.
  We define the energy
 \[
      E^n=\sqrt{\frac{1}{2} \big(\|v_h^{n}\|^2+\|\nabla u_h^{n}\|^2+\lambda\|u_h^n\|^2\big)+ (r_h^{n})^2},\ \ 1\le n\le N.
  \]
    Here $\|u\|^2=(u,u)=\|u||_{L^2}^2$.
   Taking $(w_h,\zeta_h)=(v_h^{n+1}-v_h^n, u_h^{n+1}-u_h^n)$ and multiplying the third equation  of  \eqref{sav} by ${r_h^{n+1}+r_h^n}$,
    we derive
  \begin{eqnarray*}
    \frac{1}{2} \big(\|v_h^{n+1}\|^2-\|v_h^{n}\|^2 + \|\nabla u_h^{n+1}\|^2-\|\nabla u_h^{n}\|^2 +\lambda \|u_h^{n+1}\|^2-
    \lambda\|u_h^{n}\|^2 \big)+(r_h^{n+1})^2-(r_h^{n})^2=0.
  \end{eqnarray*}
   Consequently,
 \[
      E^{n+1}=E^{n}=E^{0},\ \ \forall n\ge 1.
 \]

  Now  we  consider  a time-discrete system of equation:
 \begin{eqnarray}\label{timesystem}
 \begin{split}
    &  D_{\tau}{U^{n+1}}=\hat V^{n+\frac 12}, &\\
    & D_{\tau} {V^{n+1}}= \triangle\hat U^{n+\frac 12} -\lambda \hat U^{n+\frac 12}- \frac{\hat R^{n+\frac 12}}{\sqrt {E(\tilde U^{n+\frac 12})}}f(\tilde U^{n+\frac 12}),  &\\
    &  {R^{n+1}- R^{n}} =\frac{1}{2\sqrt {E(\tilde U^{n+\frac 12})}}\int_{\Omega} f(\tilde U^{n+\frac 12}) ({U^{n+1}-U^n})  d{\bf x}, &
 \end{split}
\end{eqnarray}
  subject to the periodic boundary condition and the following initial conditions
\begin{eqnarray*}
    U^0({\bf x})=u_0({\bf x}),\ \ V^{ 0}({\bf x})=u_1({\bf x}).
\end{eqnarray*}
   As we may observe, the numerical solution $(u_h^n,v_h^n,r_h^n)$ can be viewed as the Galerkin approximation
   of the above time-discrete system of equation.
  To study the convergence of the temporal discretization \eqref{timesystem}, we need some preliminaries.

  First, for the simplicity of notations, throughout this paper, we denote by $C$ a generic positive constant,
  which depend solely upon the physical parameters of the
problem and independent of $\tau, h, n$, and it is not necessary  to be the same at every appearance.
We  adopt the usual notations for Sobolev spaces, e.g.,
$ W^{m,p}(I)$ on sub-domain $ I\in \Omega$  equipped with the norm $\|\cdot\|_{W^{m,p},I}$
and  semi-norm $|\cdot|_{W^{m,p},I}$.
We omit the index $I$ when $I= \Omega$. Especially, when $p = 2$, we set $ W^{m,p}(I)=H^m(I)$ and
$\|\cdot\|_{W^{m,p},I}=\|\cdot\|_{H^{m},I}$ and $|\cdot|_{W^{m,p},I}=|\cdot|_{H^{m},I}.$
The notation $\alpha\lesssim \beta$ implies  that $\alpha$ is bounded by $\beta$ multiplied by a
 constant independent of $\tau, h, n$.

 Second, we would like to present a Gronwall-type inequalities, which play important role in our later convergence analysis
 and error estimates.

\begin{lemma}\label{lem1}( \cite{hey90} )
{\it Let $\tau$, $B$ and $a_k$, $b_k$, $c_k$, $\gamma_k$, for integers $k>0$, be
nonnegative numbers such that
\[  a_n+\tau\sum_{k=0}^nb_k\leq \tau \sum_{k=0}^n\gamma_ka_k+\tau\sum_{k=0}^nc_k+B,~~~~~~\mbox{for}~~n\geq 0.\]
Suppose that $\tau\gamma_k<1$, for all $k$, and set $\sigma_k=(1-\tau\gamma_k)^{-1}$. Then,
\[   a_n+\tau\sum_{k=0}^nb_k\leq \Big(\tau\sum_{k=0}^nc_k+B\Big)\exp \Big(\tau\sum_{k=0}^n\gamma_k\sigma_k \Big).\]
}
\end{lemma}

{\begin{lemma}\label{lem:2}( \cite{R1} )
{\it Let $I=[a,b]$ and $\alpha(t), \beta(t), u(t)\in C^0(I)$. Suppose $\beta(t)\ge 0$ and
\[
    u(t)\le \alpha(t)+\int_{a}^t\beta(s)u(s)ds, \ \forall t\in I.
\]
 Then
\[
    u(t)\le \alpha (t)+\int_{a}^t\alpha(s)\beta(s)e^{\int_{s}^t\beta(r) dr} ds,\ \ \forall t\in I.
\]
}
\end{lemma}
}

  Now we are ready to study the convergence of the solution of \eqref{timesystem}.
  Taking the inner product of the first two equations with
 $V^{n+1}-V^n, U^{n+1}-U^n$ and multiplying the third equation  of  \eqref{timesystem} by ${R^{n+1}+R^n}$,
    we derive
\begin{eqnarray*}
    \!\frac{1}{2} \big(\|V^{n+1}\|^2\!-\!\|V^{n}\|^2 \!+\! \|\nabla U^{n+1}\|^2-\|\nabla U^{n}\|^2 \!+\!\lambda \|U^{n+1}\|^2\!-\!
    \lambda\|U^{n}\|^2 \big)\!+\!(R^{n+1})^2\!-\!(R^{n})^2=0,
  \end{eqnarray*}
   which indicates that
\[
    \|V^{n}\|+\| U^{n}\|_{H^1}+ |R^{n}|\lesssim 1,\ \ \forall 1\le n\le N.
\]
   As pointed out in \cite{shen-xu}, energy stable is not sufficient for the convergence
   which typically needs bounds in higher norms.  Following the idea in \cite{shen-xu}, our convergence analysis is along this line:
   we first start from the energy preservation to derive the error bounds in higher norms, (i.e., the $H^2$ estimates) for the
  solution $U^n$, and thus get the $L^{\infty}$ for $U^n$ thanks to the embedding theory, and then we use the bounds in
   $H^2$ norms to show that the
   numerical solution $U^n$ converges to the exact solution $u^n$ in some suitable norms as $\tau$ tends to zero.
   To this end, we   need to the bounds in $H^2$ norm of the PDE system \eqref{con_laws}.
   The error bounds for the solution of \eqref{timesystem} is similar to those of  { the}  PDE system.

   Note that most of the convergence and error analysis for linearly implicit are based on the so called Lipschitz assumption, i.e.,
 \begin{equation}\label{lip-cod}
    |F'(u_1)-F'(u_2)|\le L |u_1-u_2|,\ \ \forall u_1,u_2.
 \end{equation}
   The above assumption greatly limits its range of applicability.
    Following the basic idea of \cite{shen-xu}, we adopt the following assumption instead of the Lipschitz assumption in our convergence analysis:
 \begin{eqnarray}\label{condition1}
   &&|f'(x)|< C (|x|^p+1),\ \ p\ge 0 \ {\rm if}\ n=1,2; \ 0<p<4 \ {\rm if}\ n=3,\\\label{condition2}
   &&|f''(x)|< C (|x|^p+1),\ \ p\ge 0 \ {\rm if}\ n=1,2; \  0<p<3 \ {\rm if}\ n=3.
 \end{eqnarray}
   It has been proved in \cite{shen-xu} that if  $f(u)$ satisfies the conditions \eqref{condition1}-\eqref{condition2}, there holds
   for some $\sigma$, where $ 0\le \sigma<1$,   such that
 \begin{eqnarray}\label{eq:9}
    \| f''(u)\|_{L^{\infty}}+ \| f'(u)\|_{L^{\infty}} \le C (1+\|\nabla\triangle u\|^{\sigma})\le \epsilon \|\nabla\triangle u\|+C_{\epsilon}
 \end{eqnarray}
  for any $\epsilon>0$ with $C_{\epsilon}$ a constant depending on $\epsilon$.

  We present the following estimates for the exact solution of \eqref{con_laws}.
\begin{proposition} Assume that $u$ is the solution of \eqref{con_laws}, and
   $u_0\in H^3, u_1\in H^2$ and \eqref{condition1}-\eqref{condition2} holds. Then for any $T>0$,
 \[
      ( \|\triangle u\| +\|\triangle u_t\|+\|\nabla\triangle u \|)(T)\lesssim 1.
 \]
\end{proposition}
\begin{proof}
  First, multiplying $ u_t$ on both sides of \eqref{con_laws} and using the integration by parts yields
\[
   \frac{d}{dt}(\|u_t\|^2+\|\nabla u\|^2+\lambda \|u\|^2+{ 2}\int_{\Omega} F(u)) =0,
\]
  which indicates that
\begin{equation}\label{eq:10}
   \|u_t\|^2+\|u\|_{H^1}^2+\int_{\Omega} F(u)\lesssim 1.
\end{equation}
  On the other hand,  we multiply $\triangle^2 u_t$ on both sides of \eqref{con_laws} and again  use the integration by parts to
  obtain
 \begin{eqnarray}
\begin{split}
    \frac{1}{2}\frac{d}{dt}\left(\|\triangle u_t\|^2+\|\nabla\triangle u \|^2+\lambda \|\triangle u\|^2\right)&={-(\triangle f(u),\triangle u_t). }&\\
\end{split}
 \end{eqnarray}
  Integrating with respect to time between $0$ to  $ t $ and using Cauchy-Schwarz inequality yields
 \begin{eqnarray}\label{eq:5}
     \left(\|\triangle u_t\|^2+\|\nabla\triangle u \|^2 +\lambda \|\triangle u\|^2 \right)(t) &\le &   \left(\|\triangle u_t\|^2+\|\nabla\triangle u \|^2+\lambda \|\triangle u\|^2\right)(0) \nonumber\\
      &+&    \int_{0}^t (\|\triangle u_t\|^2+\|\triangle f(u)\|^2)(t)dt.
 \end{eqnarray}
   By \eqref{eq:9} and the identity
\[
   \triangle f(u) = f'(u)\triangle u + f''(u)|\nabla u|^2,
\]
   we have for all $0\le \delta<1$,
\begin{eqnarray*}
    \|\triangle f(u) \|&\le& \|f''(u)\|_{L^{\infty}}\|\nabla u\|_{L^4}^2+ \|f'(u)\|_{L^{\infty}}\|\triangle u\|\\
     &\le & C (\|f''(u)\|_{L^\infty}+\|f'(u)\|_{L^{\infty}})(\|\nabla u\|_{L^4}^2+\|\triangle u\|)\\
      &\le & C(1+ \|\nabla\triangle u\|^{\delta})({\|\nabla u\|_{L^4}^2}+\|\triangle u\|).
\end{eqnarray*}
  {As for the term $\|\nabla u\|_{L^4}$, we use the Sobolev embedding theory and interpolation inequality about the spaces $H^s$ (see, e.g., \cite{T1}) and then obtain
\[
    \|\nabla u\|_{L^4}\le C \|\nabla u\|_{H^{d/4}}\le C \|\nabla u\|^{1-d/8} \|\nabla \triangle u\|^{d/8}\le C \|\nabla \triangle u\|^{d/8}.
\]
Moreover, by using the integration by parts and \eqref{eq:10},
\[
    \|\triangle u\|^2\le C \|\nabla \triangle u\|\|\nabla u\|\le C \|\nabla \triangle u\|.
\]
  Consequently,
 \[
    \|\triangle f(u) \|^2\le C(1+ \|\nabla\triangle u\|^{2\delta})(\|\nabla u\|_{L^4}^4+\|\triangle u\|^2)\le C(1+ \|\nabla\triangle u\|^{2}).
 \] }
Substituting the above inequality into
 \eqref{eq:5} gives
\begin{eqnarray}
\begin{split}
     \left(\|\triangle u_t\|^2+\|\nabla\triangle u \|^2+\lambda \|\triangle u\|^2\right)(t)&\lesssim   \left(\|\triangle u_t\|^2+\|\nabla\triangle u \|^2+\lambda \|\triangle u\|^2\right)(0)+1+ &\\
     &+\int_{0}^t\left(\|\triangle u_t\|^2+\|\nabla\triangle u \|^2\right)(t) dt  . &
\end{split}
 \end{eqnarray}
  By Gronwall inequality in Lemma \ref{lem:2}, there holds
 \begin{eqnarray*}
 \left(\|\triangle u_t\|^2+\|\nabla\triangle u \|^2+\lambda \|\triangle u\|^2\right)(t)&\lesssim   \left(\|\triangle u_t\|^2+\|\nabla\triangle u \|^2+\lambda \|\triangle u\|^2\right)(0)+1\lesssim 1.
 \end{eqnarray*}
  This finishes our proof.
\end{proof}


  Similar to the proof in the  above Proposition, we also have the following $H^2$ estimates for the  solution of \eqref{timesystem}.

\begin{proposition}
 Assume that $(U^n,V^n,R^n)$ are the solutions of \eqref{timesystem} and
 \eqref{condition1}-\eqref{condition2} holds. Then
 \begin{equation}\label{H2uh}
       \|\triangle U^n\| +\|\triangle V^n\|^2+\|\nabla\triangle U^n \|^2\lesssim 1.
 \end{equation}
\end{proposition}
\begin{proof}  First, we have from the first equation of \eqref{timesystem} that
\[
    D_{\tau}\nabla \triangle  U^{n+1}=\nabla \triangle \hat V^{n+{ \frac 12}}.
\]
Multiplying  the above equation with $\nabla  (V^{n+1}-V^{n})$ and the  second equation of \eqref{timesystem}  with $ \triangle^2 (U^{n+1}-U^n)$, and then using the
integration by parts, we obtain
\begin{eqnarray*}
     &&\|\nabla\triangle U^{n+1}\|^2- \|\nabla\triangle U^{n}\|^2+\|\triangle V^{n+1}\|^2-\|\triangle V^{n}\|^2+\lambda \| \triangle U^{n+1}\|^2-\lambda \|\triangle U^{n}\|^2\\
     &&= \frac{{ 2}\hat R^{n+\frac 12}}{\sqrt {E(\tilde U^{n+\frac 12})}} (\nabla f(\tilde U^{n+\frac 12}), \nabla\triangle (U^{n+1}-U^n)).
\end{eqnarray*}
Denoting { $f_1(\tilde U^{n+\frac 12})=\frac{ 2\hat R^{n+\frac 12}f(\tilde U^{n+\frac 12})}{\sqrt {E(\tilde U^{n+\frac 12})}} $ }
and summing up the above equation for all $n$ from $0$ to $m$ yields
\begin{eqnarray}\label{eq:6}
\begin{split}
    & \|\nabla\triangle U^{m+1}\|^2- \|\nabla\triangle U^{0}\|^2+\|\triangle V^{m+1}\|^2-\|\triangle V^{0}\|^2+\lambda \| \triangle U^{m+1}\|^2-\lambda \|\triangle U^{0}\|^2 &\\
      & =  \!(\nabla f_1(\tilde  U^{m\!+\!\frac 12}),\!  \nabla\triangle \!U^{m\!+\!1}\! )\!-\!(\nabla f_1\!(\tilde U^{\frac 12}\!) , \!\nabla\triangle\! U^{0} \!)
      \!+\! \sum_{n=1}^m \!(\nabla \! (f_1(\tilde U^{n\!-\!\frac 12})\!-\!f_1(\tilde U^{n+\frac 12})\!), \nabla\triangle U^n) &\\
      & =  (\nabla f_1(  U^{m+1}),  \nabla\triangle U^{m+1} )-(\nabla f_1(\tilde U^{\frac 12}) , \nabla\triangle U^{0} )
      + I_1-I, &
\end{split}
\end{eqnarray}
  where { $f_1( U^{m+1})=\frac{ 2\hat R^{m+\frac 12}f(  U^{m+1})}{\sqrt {E(\tilde U^{m+\frac 12})}} $ }, and
 \[
     I=  (\nabla f_1( U^{m+1})-\nabla  f_1(\tilde U^{m+\frac 12}),  \nabla\triangle U^{m+1} ),\ \ \
     I_1=\sum_{n=1}^m (\nabla  (f_1(\tilde U^{n-\frac 12})-f_1(\tilde U^{n+\frac 12})), \nabla\triangle U^n).
 \]
 { Since $R^{n}$ is bounded and $E(U)$ is bounded from below, we have
   from \eqref{eq:9} that,
\begin{equation}\label{eq:2}
   \|\nabla f_1(U^{m+1})\|\le C (1+\|f'(U^{m+1})\|_{L^{\infty}})\le \epsilon \|\nabla\triangle U^{m+1}\|+C_{\epsilon}.
\end{equation}}
 Consequently,
 \begin{equation}\label{eq:4}
   | (\nabla f_1( U^{m+1}),  \nabla\triangle U^{m+1} )-(\nabla f_1(\tilde U^{\frac 12}) , \nabla\triangle U^{0} )|\le
   \frac{1}{4}\|\nabla\triangle U^{m+1}\|^2+\frac{1}{4}\|\nabla\triangle U^{0}\|^2+C.
 \end{equation}
 On the other hand, we note that
\begin{eqnarray*}
\begin{split}
   & \nabla  f_1(\tilde U^{n-\frac 12})-\nabla  f_1(\tilde U^{n+\frac 12})&\\
   &=f'_1(\tilde U^{n-\frac 12})\nabla \tilde U^{n-\frac 12}-f'_1(\tilde U^{n+\frac 12})\nabla \tilde U^{n+\frac 12}\\
   &=f'_1(\tilde U^{n-\frac 12})(\nabla \tilde U^{n-\frac 12}-\nabla \tilde U^{n+\frac 12})+
  (f'_1(\tilde U^{n-\frac 12})-  f'_1(\tilde U^{n+\frac 12}))\nabla \tilde U^{n+\frac 12}&\\
   &=\frac{\tau}{2}f'_1(\tilde U^{n-\frac 12})\nabla (\hat V^{n-\frac 32}-3\hat  V^{n-\frac 12})+\frac{\tau}{2} f''_1(\theta \tilde U^{n+\frac 12}+(1-\theta)\tilde U^{n-\frac 12})(\hat  V^{n-\frac 32}-3\hat V^{n-\frac 12})\nabla \tilde U^{n+\frac 12}
\end{split}
\end{eqnarray*}
  for some $\theta\in (0,1)$,
  where in the last step, we have used the first equation of {\eqref{timesystem}}, which yields
 \[
    \nabla \tilde U^{n-\frac 12}-\nabla \tilde U^{n+\frac 12}={\frac 12\nabla (4U^{n-1}-3U^n-U^{n-2})
    =\frac{\tau}{2} \nabla(\hat  V^{n-\frac 32}-3\hat V^{n-\frac 12}).}
 \]
    By  \eqref{eq:2} and  the fact that
\[
    \|V^n\|+\|\nabla U^n\|\lesssim 1,
\]
we have
 \begin{eqnarray*}
      \| \nabla  f_1\!(\tilde U^{n-\frac 12}\!)\!-\!\nabla  \! f_1\!(\tilde  U^{n+\frac 12}\!)\!\|^2\!\le \!\tau^2 (\epsilon \| \nabla\triangle \tilde  U^{n\!-\!\frac 12}\|^2\!+\! \epsilon\| \nabla\triangle \tilde  U^{n+\frac 12}\|^2
      \!+\!\|  \nabla(\hat  V^{n-\frac 32}\!-3\!\hat   V^{n-\frac 12})\|^2\!+\!C^2),
 \end{eqnarray*}
   and thus
\begin{eqnarray*}
    |I_1| \le C\tau\sum_{n=1}^m   (\|\nabla\triangle U^{n}\|^2+ \|{\nabla} V^{n}\|^2) +C\tau.
\end{eqnarray*}
   Similarly, there holds
 \begin{eqnarray*}
      \| \nabla  f_1(\tilde U^{m+\frac 12})-\nabla  f_1(  U^{m+1})\|^2\le \tau^2
      \epsilon \sum_{n=m-1}^{m+1} (\|  \nabla  V^{n}\|^2)+  \| \nabla\triangle  U^{n}\|^2)+C_{\epsilon}\tau^2.
 \end{eqnarray*}
   Then
 \[
   | I| \le  \tau^2
       \sum_{n=m-1}^{m+1} (\|  \nabla  V^{n}\|^2)+ \epsilon \| \nabla\triangle  { U}^{n}\|^2)+C_{\epsilon}\tau^2+\frac{1}{4} \| \nabla\triangle  U^{m+1}\|^2.
 \]
     Substituting \eqref{eq:4}, the estimates of $I_1$ and $I$ into \eqref{eq:6} yields
\begin{eqnarray*}
    \frac 12 \|\nabla\triangle U^{m+1}\|^2+\|\triangle V^{m+1}\|^2+\lambda \|\triangle U^{m+1}\|^2&\le& \|\triangle V^{0}\|^2
    +\|\nabla\triangle U^0\|^2+\lambda \|\triangle U^{0}\|^2\\
    &+& C\tau\sum_{n=1}^{m+1}   (\|\nabla\triangle U^{n}\|^2+ \|\triangle V^{n}\|^2) +C.
    \end{eqnarray*}
  By Gronwall inequality, we have
\[
  \|\nabla\triangle U^{m+1}\|^2+\|\triangle V^{m+1}\|^2+\lambda \|\triangle U^{m+1}\|^2\le  \|\triangle V^{0}\|^2
    +\|\nabla\triangle U^0\|^2+\lambda \|\triangle U^{0}\|^2\le C.
\]
  In case $\lambda=0$, we  note that
\[
     \triangle U^{m+1}=\triangle U^{m}+\frac{\tau}{2}  (\triangle V^{{m}+1}+\triangle V^{{m}}).
\]
  Then
 \[
      \| \triangle U^{m+1}\| \le \| \triangle U^{m}\|+\frac{\tau}{2} \| \triangle  V^{m+1}+\triangle V^{m}\|,
 \]
   which yields
\[
      \| \triangle U^{m+1}\| \le \| \triangle U^{0}\|+C\le C.
\]
  This finishes the proof of \eqref{H2uh}. The proof is complete.
\end{proof}

As a direct consequence of \eqref{H2uh} and  the embedding  inequality, we have
\begin{equation}\label{inftyuh}
    \|U^n\|_{\infty}\le C \|\triangle U^{n}\| \le  C,\ \ \forall n.
\end{equation}

\begin{remark}
   Following the same argument as that in \cite{shen-xu}, we conclude that:
   Assume that $u_0\in H^3$, when $\tau $ tends to zero, we have
   $U^n\rightarrow u^n$ strongly in $L^{\infty}(0,T; H^{3-\epsilon}),\ \forall\epsilon>0$, weak-star in
 $L^{\infty}(0,T; H^{3})$,  $V^n\rightarrow v^n$  weak-star in  $L^{\infty}(0,T; H^{2})$, and
 $R^n\rightarrow r^n$ weak-star in $L^{\infty}(0,T)$;
\end{remark}

\begin{theorem}\label{theo:11}
 Suppose $u$ is the solution of \eqref{con_laws}, satisfying
 \[
      \|u_0\|_{H^{2}}+\|u\|_{L^{\infty}((0,T),H^{2})}+\|u_t\|_{L^{2}((0,T),H^{2})}+\|u_{tt}\|_{L^{2}((0,T),H^{2})}\lesssim 1.
 \]
     Then \eqref{timesystem} admits a unique solution $(U^n,V^n,R^n)$ such that
\[
    \|\ u^n-  U^n\|_{H^1}+\|v^n-V^n\|+|r^n-R^n|\lesssim \tau^2.
\]
\end{theorem}
\begin{proof}
   First we denote
 \[
     E^n_u=u^n-U^n, \ \ E^n_v=v^n-V^n,\ \ E^n_r=r^n-R^n,\ \ H(u)=\frac{f(u)}{\sqrt{E(u)}}.
 \]
   {By  taking $t=t_{n+\frac 12}$ in \eqref{con_laws1}}  and using  \eqref{timesystem},  we get
 \begin{eqnarray}\label{sav-time}
 \begin{split}
    & D_{\tau}E_u^{n+1}=\hat E_v^{n+\frac 12}+T_1, &&\\
   &D_{\tau} E_v^{n+1}=\triangle \hat E_u^{n+\frac 12}{-}\lambda \hat E_u^{n+\frac 12}
    - \hat r^{n+\frac 12} H( u^{n+\frac 12})+\hat R^{n+\frac 12}H(\tilde U^{n+\frac 12})+T_2, &&\\
    &  {E_r^{n+1}- E_r^{n}} =\frac{1}{2}\int_{\Omega} H( u^{n+\frac 12}) ({u^{n+1}-u^n})-H( \tilde U^{n+\frac 12}) ({U^{n+1}-U^n})  d{\bf x}+T_3,  &&
 \end{split}
\end{eqnarray}
 where $T_i$ denote the truncation errors,   i.e.,
 \begin{eqnarray*}
   &&T_1 =D_{\tau}u^{n+1}-u^{n+\frac 12}_t+v^{n+\frac 12}-  \hat v^{n+\frac 12},  \\
   && T_2 = \triangle (u^{n+\frac 12}-\hat u^{n+\frac 12})+D_{\tau}v^{n+1}-v^{n+\frac 12}_t{-}\lambda u^{n+\frac 12}{+}\lambda \hat u^{n+\frac 12}+ (\hat r^{n+\frac 12}-r^{n+\frac 12})H( u^{n+\frac 12}), \\
  && T_3= \tau(D_{\tau}r^{n+1}- r^{n+\frac 12}_t){-} \frac{\tau}{2}\int_{\Omega} H( u^{n+\frac 12}) \big(D_{\tau}u^{n+1}-u^{n+\frac 12}_t \big)d{\bf x}.
 \end{eqnarray*}
   Multiplying the first equation with $E_v^{n+1}-E_v^{n}$, the second equation with $E_u^{n+1}-E_u^{n}$, and
   the third equation with
   $E_r^{n+1}+E^{n}_r$ in   \eqref{sav-time}, and then summing up three equalities, we get
 \begin{eqnarray}\label{eq:22}
\begin{split}
    & \frac{1}{2} \big(\|E_v^{n+1}\|^2-\|E_v^{n}\|^2 + \|\nabla E_u^{n+1}\|^2-\|\nabla E_u^{n}\|^2+\lambda\|E_u^{n+1}\|^2-\lambda \|E_u^{n}\|^2   \big)+|E_r^{n+1}|^2-|E_r^{n}|^2  \\
    &=(I_2,DE_u^{n+1})+(T_2,DE_u^{n+1})-(T_1,DE_v^{n+1})+2(T_3+I_3)\hat E_r^{n+\frac 12},
\end{split}
  \end{eqnarray}
    where
 \begin{eqnarray}
     && I_2= {\hat r^{n+\frac 12}}(H(\tilde U^{n+\frac 12})-H( u^{n+\frac 12})),\\
     &&  I_3=\frac{1}{2}\int_{\Omega} \big(H(u^{n+\frac 12})- H( \tilde U^{n+\frac 12})\big) ({u^{n+1}-u^n})  d{\bf x}.
 \end{eqnarray}
   We next estimate the terms $T_i, I_i, i\le 3$ respectively.
   By Taylor expansion, there holds
\begin{eqnarray*}
    \|T_1\|\lesssim  \tau^2, \ \
     \|T_2\|\lesssim \tau^2,\ \
     |T_3|\lesssim \tau^3.
\end{eqnarray*}
 By  \eqref{inftyuh} and the  fact that  $f\in C^2(\mathbb R)$, we have
 \[
    | H(U^n)|+ | H'(U^n)| +  | f'(U^n)|+ | f''(U^n)|\lesssim 1,\ \ \forall n.
 \]
 Then there exists some $\theta\in (0,1)$ such that
\begin{eqnarray*}
    \| H(\tilde U^{n+\frac 12})-H( u^{n+\frac 12})\| &= &\| H'(\theta
    \tilde U^{n+\frac 12} + (1-\theta) u^{n+\frac 12})(\tilde E_u^{n+\frac 12}+\tilde u^{n+\frac 12}-u^{n+\frac 12})\|\\
      &\lesssim & \tau^2+\|\tilde E_u^{n+\frac 12}\|,
\end{eqnarray*}
  and thus
\begin{eqnarray*}
    \|I_2\|\lesssim \tau^2+\|\tilde E_u^{n+\frac 12}\|, \ \
    |I_3|\lesssim  \tau (\tau^2+\|\tilde E_u^{n+\frac 12}\|).
  \end{eqnarray*}
  Consequently,
\begin{equation}\label{I3}
    |(T_3+I_3)\hat E_r^{n+\frac 12}|\lesssim \tau |\hat E_r^{n+\frac 12}|^2+\tau^{-1}(|I_3|^2+|T_3|^2)\lesssim \tau^5+\tau ( |\hat E_r^{n+\frac 12}|^2+\|\tilde E_u^{n+\frac 12}\|^2).
\end{equation}
    Note that
\[
   \|DE_u^{n+1}\|=\tau \|\hat E_v^{n+\frac 12}+T_1\|\lesssim \tau^3+\tau \|\hat E_v^{n+\frac 12}\|.
\]
   Then
\begin{equation}\label{I2}
   |(I_2+T_2,DE_u^{n+1})|\lesssim \tau \|I_2\|^2+\tau \|T_2\|^2+\tau^{-1}\|DE_u^{n+1}\|^2\lesssim\tau^5+\tau (\|\hat E_v^{n+\frac 12}\|^2+\|\tilde E_u^{n+\frac 12}\|{^2}).
\end{equation}
  On the other hand, in light of  the second equation of \eqref{sav-time}, we have
 \begin{eqnarray*}
    (T_1,DE_v^{n+1})&=&\tau(T_1,\triangle \hat E_u^{n+\frac 12}{-}\lambda \hat E_u^{n+\frac 12}
    - \hat r^{n+\frac 12} H( u^{n+\frac 12})+\hat R^{n+\frac 12}H(\tilde U^{n+\frac 12})+T_2)\\
    &=&\!-\!\tau(\nabla T_1, \!\nabla \hat E_u^{n+\frac 12})\!+\!\tau (T_1, {-}\lambda \hat E_u^{n\!+\!\frac 12}\!-\! \hat r^{n\!+\!\frac 12} H( u^{n\!+\!\frac 12})\!+\!\hat R^{n\!+\!\frac 12}
    H(\tilde U^{n\!+\!\frac 12})\!+\!T_2).
 \end{eqnarray*}
    Noticing that
 \begin{eqnarray*}
   \|\hat r^{n+\frac 12} H( u^{n+\frac 12}){-}\hat R^{n+\frac 12}H(\tilde U^{n+\frac 12})\|
   &=&\|\hat E_r^{n+\frac 12} H( u^{n+\frac 12})+\hat R^{n+\frac 12}(H( u^{n+\frac 12})-H(\tilde U^{n+\frac 12}))\|\\
   &\lesssim & |E_r^{{n}+\frac 12}|+\tau^2+\|\tilde E_u^{n+\frac 12}\|,
\end{eqnarray*}
    we have
\begin{equation}\label{I1}
   |(T_1,DE_v^{n+1})|\lesssim \tau^3 (\|\nabla \hat E_u^{n+\frac 12}\|+\|\hat E_u^{n+\frac 12}\|^2+\tau^2+|E_r^{{n}+\frac 12}|+\|\tilde E_u^{n+\frac 12}\|).
\end{equation}
  Plugging \eqref{I3}-\eqref{I1} into \eqref{eq:22} yields
\begin{eqnarray*}
  &&\frac{1}{2} \big(\|E_v^{n\!+\!1}\|^2\!-\!\|E_v^{n}\|^2\! +\! \|\nabla E_u^{n\!+\!1}\|^2\!-\!\|\nabla E_u^{n}\|^2\!+\!\lambda\|E_u^{n+1}\|^2\!-\!\lambda \|E_u^{n}\|^2   \big)\!+\!(E_r^{n+1})^2\!-\!(E_r^{n})^2\\
  &&\lesssim   \tau^5+ \tau ( \|\tilde E_u^{n+\frac 12}\|^2 +\|\hat E_v^{n+\frac 12}\|^2 + \|\nabla \hat E_u^{n+\frac 12}\|^2+\|\hat E_u^{n+\frac 12}\|^2 + |\hat E_r^{n+\frac 12}|^2 \big).
\end{eqnarray*}
   Summing up all $n$ from $0$ to $m$ and using the initial values, we get
\begin{eqnarray*}
  && \|E_v^{m+1}\|^2+ \|\nabla E_u^{m+1}\|^2+\lambda\|E_u^{m+1}\|^2+|E_r^{m+1}|^2\\
  && \le C\tau^4+C \tau \sum_{n=0}^{m+1}(\|E_v^{n}\|^2+ \|\nabla E_u^{n}\|^2+\lambda\|E_u^{n}\|^2+|E_r^{n}|^2).
\end{eqnarray*}
 Then the desired result follows  from the  conclusion in Lemma \ref{lem1}.

 \end{proof}

 \section{Error estimates for the fully discrete solution}\label{error-es}


  In this section, we establish the error estimates for the solution of \eqref{sav}.
  By the error decomposition, we have
\[
   w^n-w_h^n=w^n-W^n+W^n-w_h^n=w^n-W^n+e_w^n,\ \ w=u,v,r.
\]
  In light of the conclusion in Theorem \ref{theo:11}, we only need to estimate the term $e_w, w=u,v,r$.
  To this end, we first
 define the Ritz projection operator
$R_h: H_0^1(\Omega)\rightarrow V_h$ by
\begin{eqnarray*}
\left(\nabla(v-R_h v),\nabla\omega \right)=0,  \quad \forall \omega\in V_h.
\end{eqnarray*}
  Then $e^n_u$ (similar for $e_v^n$) can be decomposed into
\[
   e_u=U^n-u^n_h=\xi^n_u+\eta^n_u,\ \ \xi^n_u=R_hU^n-u^n_h,\ \ \eta_u=U^n-R_hU^n.
\]
   According to the standard FEM theory \cite{Th}, it holds that
\begin{eqnarray}\label{Rh}
\|v-R_h v\|_{L^2}+h\|\nabla (v-R_h v)\|_{L^2}\leq C h^{s}\|v\|_{H^s}, \quad \forall v\in H^s(\Omega)\cap H_0^1(\Omega)
\end{eqnarray}
for $1\leq s\leq k+1$.

 Note that the exact solutions of \eqref{timesystem} satisfy
  \begin{eqnarray}\label{sav1}
 \begin{split}
    & \big(D_{\tau}U^{n+1}, w_h\big)=\big(\hat V^{n+\frac 12}, w_h\big), &&\\
    &\big(D_{\tau}V^{n+1}, \zeta_h\big)={-}\big( \nabla \hat U^{n+\frac 12}, \nabla \zeta_h\big){-}\lambda(\hat U^{n+\frac 12},\zeta_h)
    -\big( {\hat R^{n+\frac 12}}H( \tilde U^{n+\frac 12}), \zeta_h \big), &&\\
    &  {R^{n+1}-R^{n}} =\frac{1}{2}\int_{\Omega} H(\tilde U^{n+\frac 12}) ({U^{n+1}-U^n})  d{\bf x},  &&
 \end{split}
\end{eqnarray}
  Subtracting \eqref{sav} from  \eqref{sav1} gives the following error equation
   \begin{eqnarray}\label{sav2}
 \begin{split}
    & \big(D_{\tau} \xi_u^{n+1}, w_h\big)=\big(\hat \xi_v^{n+\frac 12}, w_h\big)+R_1(w_h),&&\\
    &\big(D_{\tau} \xi_v^{n+1}, \zeta_h\big)=-\big( \nabla \hat \xi_u^{n+\frac 12}, \nabla \zeta_h\big){-}\lambda \big(  \hat \xi_u^{n+\frac 12}, \zeta_h\big)
    -\big( {\hat e_r^{n+\frac 12}}H(\tilde u_h^{n+\frac 12}), \zeta_h \big)+R_2(\zeta_h)+I_2(\zeta_h), &&\\
    &  {e_r^{n+1}- e_r^{n}} =\frac{1}{2}\int_{\Omega} H( \tilde u_h^{n+\frac 12}) ({\xi_u^{n+1}-\xi_u^n})  d{\bf x}+R_3+I_3,  &&
 \end{split}
\end{eqnarray}
  where $H(u)= \frac{f(u)}{\sqrt {E( u) }}$, and
 \begin{eqnarray}\label{eq:21}
 \begin{split}
     &&R_1(w_h)= \big(\hat \eta_v^{n+\frac 12}, w_h\big)-\big(D_{\tau} \eta_u^{n+1}, w_h\big),\\
     &&R_2(\zeta_h)= -\big( \nabla \hat \eta_u^{n+\frac 12}, \nabla \zeta_h\big){-}\lambda \big(  \hat \eta_u^{n+\frac 12}, \zeta_h\big)- \big(D_{\tau}\eta_v^{n+1}, \zeta_h\big),\\
     && I_2(\zeta_h)=\big( {\hat R^{n+\frac 12}}(H(\tilde u_h^{n+\frac 12})-H( \tilde U^{n+\frac 12})) , \zeta_h \big),\\
     && R_3= \frac{1}{2}\int_{\Omega} H( \tilde u_h^{n+\frac 12}) ({\eta_u^{n+1}-\eta_u^n})  d{\bf x},\\
     && I_3=\frac{1}{2}\int_{\Omega} \big(H(\tilde U^{n+\frac 12})- H( \tilde u_h^{n+\frac 12})\big) ({{U}^{n+1}-U^n})  d{\bf x}.
 \end{split}
 \end{eqnarray}

\begin{theorem}\label{theorem:1}
   Suppose $u$ is the solution of \eqref{con_laws} satisfying
 \[
      \|u_0\|_{H^{k+1}}+\|u\|_{L^{\infty}((0,T),H^{k+1})}+\|u_t\|_{L^{2}((0,T),H^{k+1})}+\|u_{tt}\|_{L^{2}((0,T),H^{k+1})}\lesssim 1,
 \]
    and $(u_h^n,v_h^n,r_h^n)$ is the solution of \eqref{sav} with
  { $(u_h^0,v_h^0,r_h^0)=(R_h u_0,R_h u_1,\sqrt{E(u_0)})$.}  Then
 \begin{eqnarray}\label{opt:1}
        \|R_hv^n-v_h^{n}\|+ \|R_h u^n-  u_h^{n}\|_{H^1} + |r^n-r_h^n|\lesssim h^{k+1}+\tau^2.
 \end{eqnarray}
\end{theorem}
\begin{proof} We first estimate the terms $\xi_u^n,\xi_v^n$.
 By taking   $(w_h,\zeta_h)=(D\xi_v^{n+1}, D\xi_u^{n+1})$ and multiplying the third equation  of  \eqref{sav2} by $2\hat e_r^{n+\frac 12}={e_r^{n+1}+e_r^n}$, we derive
  \begin{eqnarray*}
    && \frac{1}{2} \big(\|\xi_v^{n+1}\|^2-\|\xi_v^{n}\|^2 + \|\nabla \xi_u^{n+1}\|^2-\|\nabla \xi_u^{n}\|^2+ \lambda \| \xi_u^{n+1}\|^2-\| \lambda \xi_u^{n}\|^2  \big)+(e_r^{n+1})^2-(e_r^{n})^2  \\
    &=&R_2(D\xi_u^{n+1})+I_2(D\xi_u^{n+1})-R_1(D\xi_v^{n+1})+2(R_3+I_3)\hat e_r^{n+\frac 12},
  \end{eqnarray*}
    where $R_i, I_i, i\le 3$ are given in \eqref{eq:21}. Summing up all $n$ from $0$ to $m$ and using the initial
    error $\xi_u^0=\xi_v^0=0$, we get
 \begin{eqnarray}\label{3.4}
 \begin{split}
    & \frac{1}{2}(\|\xi_v^{{m}+1}\|^2 + \|\nabla \xi_u^{m+1}\|^2+ \lambda  \| \xi_u^{m+1}\|^2)+|e_r^{m+1}|^2 & \\
   &=\sum_{n=0}^m \left( R_2(D\xi_u^{n+1})+I_2(D\xi_u^{n+1})-R_1(D\xi_v^{n+1})+2(R_3+I_3)\hat e_r^{n+\frac 12}\right). &
 \end{split}
  \end{eqnarray}

  To estimate the terms in the right hand side of \eqref{3.4},  we shall first make the  hypothesis that there exists a positive constant $C_*$ such that
\begin{equation}\label{hypo:uh}
   \|u_h^n\|_{L^{\infty}}\le C_*.
\end{equation}
  This hypothesis will be verified later by using the method of mathematical induction.

  Due to \eqref{hypo:uh} and the  fact that  $f\in C^2(\mathbb R)$, we have
 \[
    | H(u_h^n)|+ | H'(u_h^n)| +  | f'(u_h^n)|+ | f''(u_h^n)|\lesssim 1,\ \ \forall n.
 \]
  Then
\[
   |R_3|\lesssim \tau h^{k+1},\ \ |R_2(D\xi_u^{n+1})|\lesssim  h^{k+1} \|D\xi_u^{n+1}\|.
\]
By Taylor expansion, there exists a $\theta\in (0,1)$ such that
\begin{eqnarray}\nonumber
    \| H(\tilde u_h^{n+\frac 12})-H(\tilde U^{n+\frac 12})\| &=& \| H'(\theta
    \tilde u_h^{n+\frac 12} + (1-\theta) \tilde U^{n+\frac 12})(\tilde u_h^{n+\frac 12}-\tilde U^{n+\frac 12})\| \\ \label{R2}
     &\lesssim &  h^{k+1}+\|\tilde \xi_u^{n+\frac 12}\|.
\end{eqnarray}
  Then
\begin{eqnarray*}
   |I_2(D\xi_u^{n+1})|\lesssim (h^{k+1}+\|\tilde \xi_u^{n+\frac 12}\|)\|D\xi_u^{n+1}\|,\ \
    |I_3|\lesssim \tau (h^{k+1}+\|\tilde \xi_u^{n+\frac 12}\|),
  \end{eqnarray*}
  and thus
\begin{equation}\label{R3}
   \big| \sum_{n=0}^m 2(R_3+I_3)\hat e_r^{n+\frac 12}\big|\le C h^{2(k+1)}+\tau\sum_{n=1}^m \big(\|\tilde \xi_u^{n+\frac 12}\|^2+ |\hat e_r^{n+\frac 12}|^2\big).
\end{equation}
  On the other hand,
  we choose $w_h=D\xi_u^{n+1}$ in \eqref{sav2} to get
\[
   \|D\xi_u^{n+1}\|\lesssim \tau  (\| \hat \xi_v^{n+\frac 12}\|+ h^{k+1}),
\]
   which yields, together with Cauchy-Schwarz inequality
\[
   |I_2(D\xi_u^{n+1})|+|R_2(D\xi_u^{n+1})|\lesssim \tau h^{2(k+1)}+\tau(\| \hat \xi_v^{n+\frac 12}\|^2+\|\tilde \xi_u^{n+\frac 12}\|^2).
\]
 Consequently,
\begin{equation}\label{eq:1}
   \sum_{n=0}^m|I_2(D\xi_u^{n+1})|+|R_2(D\xi_u^{n+1})|\lesssim  h^{2(k+1)}+\tau\sum_{n=0}^m (\| \hat \xi_v^{n+\frac 12}\|^2+\|\tilde \xi_u^{n+\frac 12}\|^2).
\end{equation}
As for the term  $\sum_{n=0}^mR_1(D\xi_v^{n+1})$ in \eqref{3.4}, we recall the definition of $R_1$ in \eqref{eq:21} to get
\begin{eqnarray*}
   \big| \sum_{n=0}^mR_1(D\xi_v^{n+1})\big|&=&\big| \big(\hat \eta_v^{m+\frac 12}-D_{\tau} \eta_u^{m+1}, \xi_v^{m+1}\big)
   +\sum_{n=1}^m\big(\hat \eta_v^{n-\frac 12}-\hat \eta_v^{n+\frac 12}+D_{\tau} \eta_u^{n+1}-D_{\tau} \eta_u^{n}, \xi_v^{n}\big) \big|\\
    &\le & C h^{k+1}\|\xi_v^{m+1}\|+C\tau h^{k+1}\sum_{n=1}^m\|\xi_v^{n}\|\\
    &\le& C h^{2(k+1)}+ C\tau \sum_{n=1}^m\|\xi_v^{n}\|^2 +\frac{1}{4}\|\xi_v^{m+1}\|^2.
\end{eqnarray*}
   Substituting the above inequality, \eqref{R3}-\eqref{eq:1} into \eqref{3.4},  we get
\[
      \|\xi_v^{m+1}\|^2 + \|\nabla \xi_u^{m+1}\|^2+   \| \xi_u^{m+1}\|^2+|e_r^{m+1}|^2
     \lesssim  h^{2(k+1)}+\tau \sum_{n=0}^{m+1}\big(\|\xi_v^{n}\|^2+|  e_r^{n}|^2+\| \xi_u^{n}\|^2\big).
\]
   By the Gronwall inequality given in Lemma \ref{lem1},
\begin{equation}\label{eq:3}
   \|\xi_v^n\|+ \|\xi_u^n\|_{H^1} + |e_r^n|\lesssim h^{k+1},\ \ \forall n\ge 1.
\end{equation}
  Then from the triangle inequality and the conclusion in Theorem \ref{theo:11},
\begin{eqnarray*}
  && \|R_hv^n-v_h^{n}\|+ \|R_h u^n-u_h^{n}\|_{H^1} + |r^n-r_h^n|\\
  && \le \|\xi_v^n\|+ \|\xi_u^n\|_{H^1} + |e_r^n|+ \|R_h (v^n-V^n)\|
   +\|R_h (u^n-U^n)\|+|r^n-R^n|\lesssim
    h^{k+1}+\tau^2.
\end{eqnarray*}
   The proof is complete.
\end{proof}

\begin{remark} { Note that the optimal convergence rate for the $H^1$ error approximation is $O(h^k)$.}
  The error estimate in \eqref{opt:1} indicates that
   the Galerkin SAV solution $u_h^n$ is { superclose} to the Ritz projection of the exact solution $R_hu^n$ under the $H^1$ norm, {which  is one order higher than the counterpart
   optimal converge rate.}
   {As} a direct consequence of \eqref{opt:1}, we have the following  optimal error estimates:
\begin{eqnarray*}
   && \|v^n-v_h^{n}\|_{H^1}+ \| u^n-  u_h^{n}\|_{H^1} \lesssim h^{k}+\tau^2,\\
    && \|v^n-v_h^{n}\|+ \| u^n-  u_h^{n}\|+|r^n-r_h^n| \lesssim h^{k+1}+\tau^2.
\end{eqnarray*}
\end{remark}

  To end with this section, we prove the inequality \eqref{hypo:uh}.

\begin{lemma}\label{lem2}
   Under  the conditions of Theorem \ref{theorem:1},  it holds that
\[
    \|u_h^n\|_{L^{\infty}}\le C_*,\ \ \forall n\ge 1,
\]
where the constant $C_*$ is independent of $\tau$ and $h$.
\end{lemma}
\begin{proof}
   We will show the above inequality by induction.
   To this end, we first denote by ${\cal I}_hU\in V_h$ the interpolation function of $U$.
   By the approximation theory, we have
\[
    \|{\cal I}_hU-U\|+\|{\cal I}_hU-R_hU\|\lesssim h^2 \|U\|_{H^2},\ \ \|{\cal I}_hU\|_{L^{\infty}}\lesssim \|U\|_{L^{\infty}}.
\]
   Note that
 \[
    \|u_h^0\|_{L^{\infty}}= \|R_hu^0\|_{L^{\infty}}\le C.
 \]
    By \eqref{eq:3} and the inverse inequality $\|v_h\|_{L^{\infty}}\lesssim h^{-\frac  d2}\|v_h\|$ for all $v_h\in V_h$, we get
\begin{eqnarray*}
   \|u_h^1\|_{L^{\infty}}&\le & \|u_h^1-R_hU^1\|_{L^{\infty}}+\|R_hU^1-{\cal I}_hU^1\|_{L^{\infty}}+\|{\cal I}_hU^1\|_{L^{\infty}}\\
   &\le &
   C ( h^{2-\frac{d}{2}} +\|U^1\|_{L^{\infty}})
  \le  C_1 ( h^{2-\frac{d}{2}} +\|U^1\|_{H^2}).
\end{eqnarray*}
   Now we choose a positive constant $h_1$ which is small enough to satisfy
\[
   C_1h_1^{\frac 12}\le C.
\]
  Then for $h\in (0,h_1]$, we derive that
\[
   \|u_h^1\|_{L^{\infty}}\le C+\|U^1\|_{H^2}\le C_3.
\]
   Therefore, we can choose the positive constant $C_*$  independent of $h$ and $\tau$
  such that
\[
    C_*\ge \max\{2\|U^n\|_{H^2},\|u_h^1\|_{L^{\infty}}\}.
\]
Then \eqref{hypo:uh} is valid for $n=1$.
Next, suppose \eqref{hypo:uh} holds for all $l\le n-1$. We will show that it is also valid for $n$.
 Thanks to \eqref{eq:3}, we have
\[
    \|u_h^n-R_hU^n\|_{H^1}\le C h^{k+1}.
\]
  Then
\begin{eqnarray*}
   \|u_h^n\|_{L^{\infty}}&\le & \|u_h^n-R_hU^n\|_{L^{\infty}}+\|R_hU^n-{\cal I}_hU^n\|_{L^{\infty}}+\|{\cal I}_hU^n\|_{L^{\infty}}\\
  &\le& C h^{-\frac{d}{2}}(\|u_h^n-R_hU^n\|+\|R_hU^n-{\cal I}_hU^n\| )+\|U^n\|_{L^{\infty}}
  \le  C_1 h^{\frac 12}+\frac{C_*}{2}.
\end{eqnarray*}
  Let  $h_1$ be small enough to satisfy
\[
   C_1 h_1^{\frac 12} \le \frac{C_*}{2}.
\]
  Then for $h\in (0,h_1]$, we derive that
\[
   \|u_h^n\|_{L^{\infty}}\le  C_1 h^{\frac 12}+\frac{C_*}{2}\le C_*.
\]
 This completes the induction.
\end{proof}


\section{Numerical simulations}\label{num}
We present several numerical results to confirm  our theoretical findings in this section.

{\bf Example 1}  Consider the following Klein-Gordon equation
\begin{eqnarray}\label{5.1}
u_{tt}=  u_{xx}+u_{yy}+ u-u^3+g_1(x,y,t). ~~~~~~(x,y,t)\in [0,1]^2 \times [0,T],
\end{eqnarray}
where $u(x,y,0)$, $u_t(x,y,0)$ and $g_1(x,y,t)$ { are given by} the exact solution
\begin{align}
  u(x,t) = \exp(-t)x^2(1-x)^2y^2(1-y)^2.
\end{align}
 We test  convergence orders of the fully discrete scheme using uniform triangulation with $M+1$ nodes in each spatial direction, and take $N=M$  and  $N=M^{\frac{3}{2}}$ for the linear finite element method (L-FEM) and quadratic finite element method (Q-FEM), respectively. We list the errors at time $T=1$ as well as the convergence rates in Table \ref{table1}.
Here and below, we denote
\[ \|e\|_0= \|u^N-u_h^N\|,~~~\|e\|_1=\| R_hu^N-u_h^N\|_{H^1}. \]
These results indicate that the fully discrete scheme is convergent and has the order $\mathcal{O}(\tau^2 + h^{r+1})$.
We also test the unconditional convergence of the fully discrete scheme with different spatial step-sizes for every fixed $\tau$. The $l_2$-errors at time $T = 1$ are shown
in Figure \ref{fig10}. When the temporal stepsize is fixed, the $L_2$ errors tend to a constant. They imply that the error estimates hold without any temporal mesh sizes restrictions dependent on the spatial mesh sizes.

 \begin{figure}[!ht]
 \includegraphics[scale=.48]{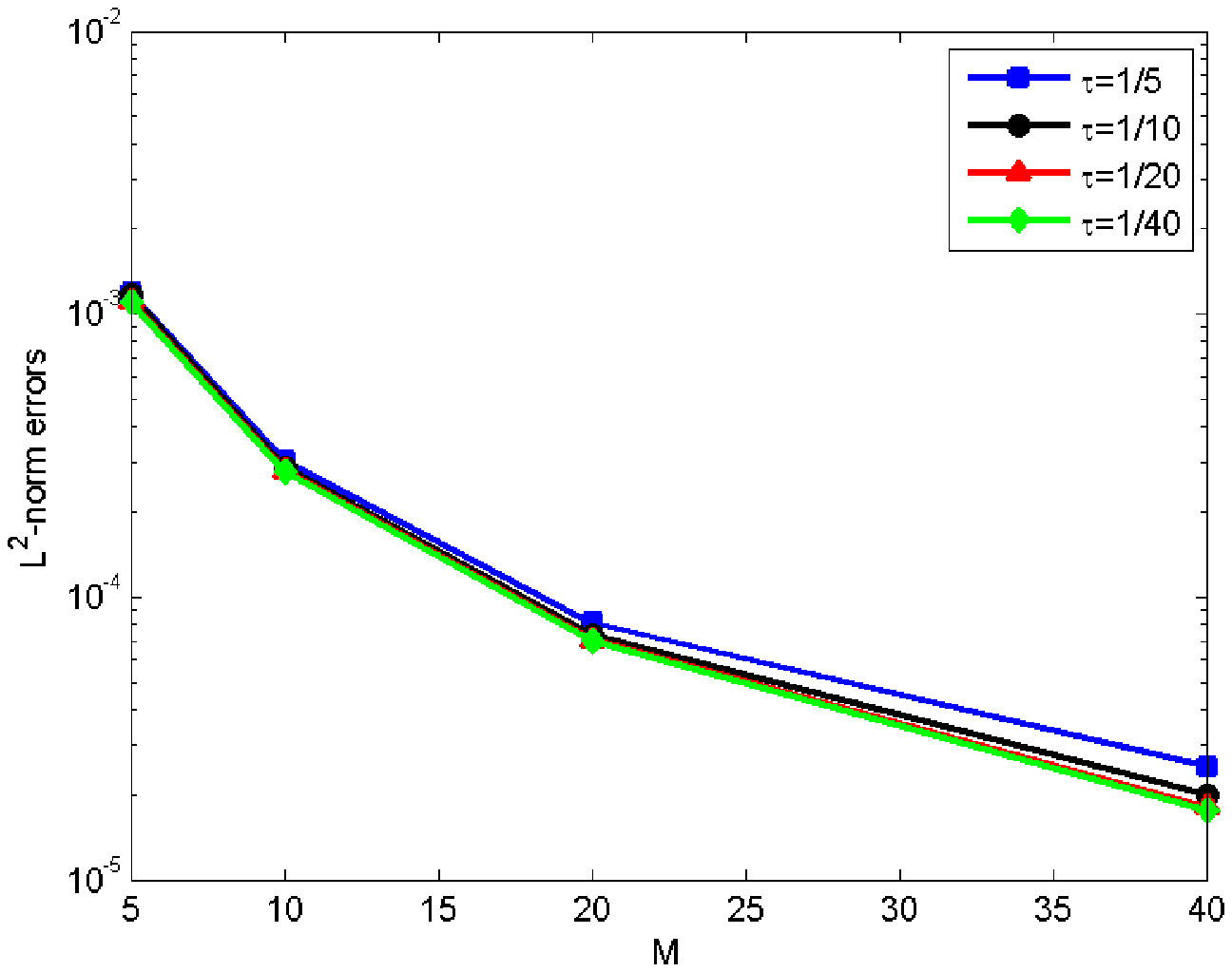}
  \includegraphics[scale=.48]{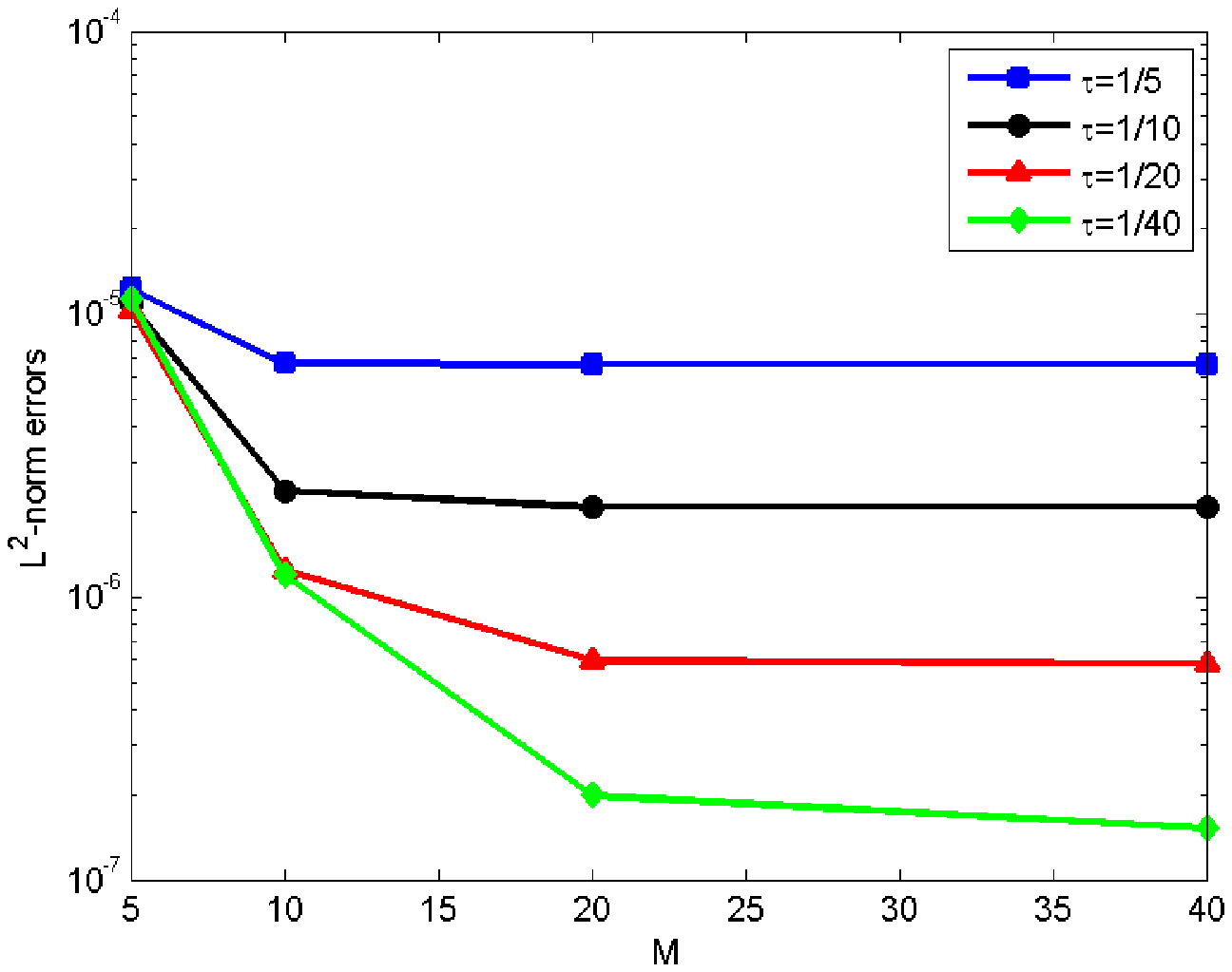}
 \caption{$L^2$-errors of linear and quadratic finite element approximation}
  \label{fig10}
 \end{figure}

\begin{table}[!ht]\caption{ Errors and convergent orders for 2D problems.}\label{table1} \centering
        \begin{tabular}{c c c c c| c cccc  }
        \hline
         & &\multicolumn{2}{c}{$\mbox{L-FEM}$}& &{}&\multicolumn{2}{c}{$\mbox{Q-FEM}$}&{}\\
\cline{2-5}\cline{6-9}
     M     &  $\|e\|_0$ & order &$\|e\|_1$ & order& $\|e\|_0$&   order &$\|e\|_1$ &order \\
        \hline\cline{1-9}
    8 &  4.54E-4 &    --  &  5.56e-4  &      -- &  2.47E-6  &    --  &  4.60E-5 &     --\\
   16 &  1.11E-4 &  2.03  &  1.61E-4  &    1.83 &  2.71E-7  &   3.19 &  6.85E-6&   2.74\\
   24 &  4.90e-5 &  2.02  &  7.40E-5  &    1.96 &  7.66E-8  &   3.12 &  2.41E-6 &  2.57\\
   32 &  2.74e-5 &  2.02  &  4.24E-5  &    1.96 &  3.18E-8  &   3.05 &  9.81E-7 &  3.10\\
   40 &  1.75E-5 &  2.01  &  2.73E-5  &    1.97 &  1.62E-8  &   3.02 &  4.88E-7 &  3.12\\
\hline
   \end{tabular}
\end{table}

\begin{figure}[!t]
\centering
\includegraphics[width=0.58\textwidth]{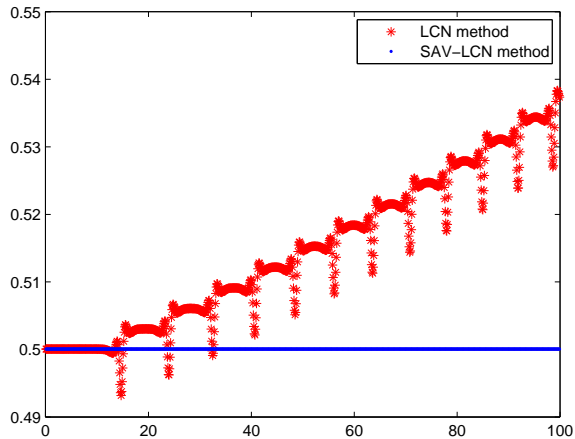}
\caption{The evolutions of the discrete energies}\label{fig11}
\end{figure}

Then, we set $g_1(x,y,t)=0$, $T=100$, $N=10$, $M=10$, and solve the problem using L-FEM.
The time discretization is achieved by the linearized Crank-Nicolson (LCN)  method, and by
the proposed SAV LCN method, respectively. The evolutions of the discrete energies are shown in Figure \ref{fig11}.  Clearly, the energies obtained by  the
LCN finite element method increase as time goes on, while the one
obtained by our method remains the same. It implies that
numerical solutions by the SAV approach conserve the energy.

\vskip.1in

{\bf Example 2}  Consider the following Sine-Gordon equation
\begin{eqnarray}\label{5.2}
u_{tt}=  u_{xx}+u_{yy}+u_{zz} +\sin(u)+g_2(x,y,z,t),~~~~~~(x,y,z,t)\in [0,1]^3 \times [0,1],
\end{eqnarray}
where the initial conditions and $g_2(x,y,z,t)$ { are produced} from the exact solution
\begin{align}
  u(x,t) = (1+t^3)\sin(2\pi x)\sin(2\pi y)\sin(2\pi z).
\end{align}

\begin{table}[!ht]\caption{ Errors and convergent orders for 3D problems.}\label{table2} \centering
        \begin{tabular}{c c c c c| c c cccc  }
        \hline
         & &\multicolumn{2}{c}{$\mbox{L-FEM}$}& & &{}&\multicolumn{2}{c}{$\mbox{Q-FEM}$}&{}\\
\cline{2-5}\cline{7-10}
     M     &  $\|e\|_0$ & order &$\|e\|_1$ & order& M & $\|e\|_0$&   order &$\|e\|_1$ &order \\
        \hline\cline{1-9}
   16 &  6.36E-2 &    --  &  4.43E-1  &      -- & 10  & 5.69E-3 &     --  &  1.29E-1 &     --\\
   20 &  4.17E-2 &  1.89  &  2.94E-1  &    1.83 & 12  & 3.24E-3  &   3.09 &  7.41E-2 &   2.96\\
   24 &  2.94E-2 &  1.92  &  2.08E-1  &    1.91 & 14  & 1.99E-3  &   3.16 &  4.91E-2 &   2.67\\
   28 &  2.18E-2 &  1.94  &  1.55E-1  &    1.90 & 16  & 1.31E-3  &   3.13 &  3.40E-2  &   2.75\\
\hline
   \end{tabular}
\end{table}


 \begin{figure}[!t]
\centering
\includegraphics[width=0.58\textwidth]{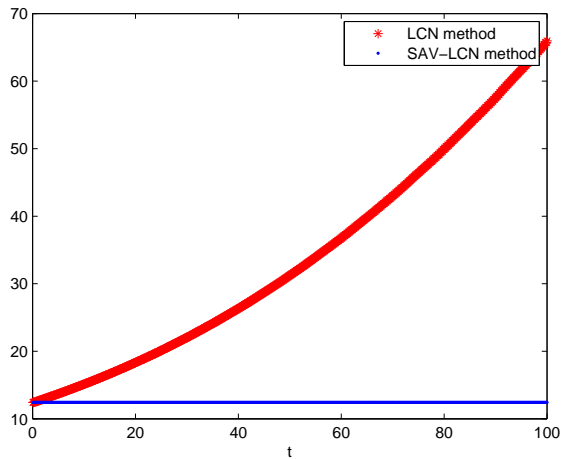}
\caption{The evolutions of the discrete energies}\label{fig21}
\end{figure}

We still take $N=M$  and  $N=M^{\frac{3}{2}}$ for the linear and quadratic finite element approximations, respectively. The numerical errors at time $T=1$ as well as the convergence rates are presented in Table \ref{table2}.
The given results indicate that the fully discrete scheme has the order $\mathcal{O}(\tau^2 + h^{r+1})$.


Next, we set $g_2(x,y,z,t)=0$, $T=20$, $N=10$, $M=10$ and solve the problem  by the linear finite element method.
The evolutions of the discrete energies for the 3D problems are displayed in Fig. \ref{fig21}.  Clearly, the discrete energies by the SAV approach remain unchanged, while the ones obtained by the LCN finite element method
 increase as time goes on. They further confirm the findings in this study.

\section{Conclusion}\label{conc}
In this study, we present a linearly implicit numerical schemes for solving the nonlinear wave equation \eqref{con_laws}.
 The scheme is developed by combining the SAV approach with finite element methods,
classical Crank-Nicolson methods, and extrapolation approximation. The fully discrete scheme is proved to be unconditionally convergent and energy-conserving.
Numerical illustrations are presented to confirm the theoretical findings.

\end{document}